\documentclass{amsart}

\usepackage{amsthm,amssymb,verbatim}
\usepackage{graphicx}
\usepackage{enumerate}

\newcommand{\HH}{\mathcal H} 
\newcommand{\LL}{\mathcal L} 
\newcommand{\MM}{\mathcal M}

\newcommand{\vv}{\bold v}

 \newcommand{\FF}{\mathcal{F}}
 \newcommand{\RR}{\mathbf{R}}  
 \newcommand{\ZZ}{\mathbf{Z}}  
 \newcommand{\BB}{\mathbf{B}}  

 \newcommand{\eps}{\epsilon}

 \newcommand{\M}{\operatorname{M}} 

\input epsf
\def\begfig {
\begin{figure}
\small }
\def\endfig {
\normalsize
\end{figure}
}

\swapnumbers 

    \newtheorem{theorem}    {Theorem}       [section]
    \newtheorem{lemma}      [theorem]       {Lemma}
    \newtheorem{corollary}  [theorem]     {Corollary}

    \theoremstyle{definition}
    \newtheorem{definition}  [theorem] {Definition}

    \theoremstyle{definition}
    \newtheorem{remark}   [theorem]       {Remark}

\usepackage{hyperref}
\usepackage[alphabetic]{amsrefs}

\begin{document}

\renewcommand{\thesubsection}{\thetheorem}

\title[Currents and Flat Chains]{Currents and Flat Chains \\Associated to Varifolds, 
     with an \\ Application to Mean Curvature Flow}
\author{Brian White}
\thanks{This research was supported by the NSF 
  under grants~DMS-0406209 and~DMS-0707126}
\email{white@math.stanford.edu}
\date{May 13, 2008.  Revised November 8, 2008}

\begin{abstract}
 We prove under suitable hypotheses that convergence of integral varifolds
 implies convergence of associated mod $2$ flat chains and subsequential
 convergence of associated integer-multiplicity rectifiable currents.  
 The convergence results imply restrictions on the kinds of singularities
 that can occur in mean curvature flow.
\end{abstract}

\subjclass[2000]{Primary: 49Q15; Secondary: 53C44}

\maketitle

\section{Introduction}\label{section:intro}

\newcommand{\LLL}{\LL_{\text{$m$-rec}}}

Let $U$ be an open subset of $\RR^N$.  Let $\LLL(U, \ZZ^+)$ denote the space of functions
on $U$ that take values in nonnegative integers, that are locally $\LL^1$ with respect to Hausdorff $m$-dimensional
measure on $U$, and that vanish except on a countable disjoint union of $m$-dimensional $C^1$ submanifolds of $U$. 
We identify functions that agree except on a set of Hausdorff $m$-dimensional measure zero.
 Let $\LLL(U,\ZZ_2)$ be the corresponding space with the nonnegative integers
$\ZZ^+$ replaced by $\ZZ_2$, the integers mod~$2$.

The space of $m$-dimensional integral varifolds in $U$ is naturally isomorphic to $\LLL(U,\ZZ^+)$:
given any such varifold $V$, the corresponding function is the density function $\Theta(V,\cdot)$
given by
\[
   \Theta(V,x) = \lim_{r\to 0}\frac{\mu_V(\BB(x,r))}{\omega_mr^m}
\]
where $\mu_V$ is the radon measure on $U$ determined by $V$ and $\omega_n$ is the volume of the
unit ball in $\RR^m$.  In particular, this limit exists and is a nonnegative integer for $\HH^m$-almost every $x\in U$.

Similarly, the space of $m$-dimensional rectifiable mod $2$ flat chains in $U$ is naturally isomorphic
to $\LLL(U, \ZZ_2)$: given any such flat chain $A$, the corresponding function is the density function
$\Theta(A,\cdot)$ given by
\[
   \Theta(A,x) = \lim_{r\to 0}\frac{\mu_A(\BB(x,r))}{\omega_mr^m} 
                       = \lim_{r\to 0}\frac{M(A\cap \BB(x,r))}{\omega_mr^m}
\]
where $\mu_A$ is the radon measure on $U$ determined by $A$.  In particular, this limit exists and is $0$ or $1$
for $\HH^m$-almost every $x\in U$.

The surjective homomorphism 
\begin{align*}
    [\cdot]: \, &\ZZ^+ \to \ZZ_2 \\
    &k \mapsto [k]
\end{align*}
determines a homomorphism from $\LLL(U,\ZZ^+)$
to $\LLL(U, \ZZ_2)$, and thus also a homomorphism from the additive semigroup of integral varifolds in $U$
to the additive group of rectifiable mod $2$ flat chains in $U$.  If $V$ is such a varifold, we let $[V]$
denote the corresponding rectifiable mod $2$ flat chain.  Thus $[V]$ is the unique rectifiable mod $2$ flat chain in $U$ such that
\[
  \Theta([V], x) = [ \Theta(V,x)]
\]
for $\HH^m$-almost every $x\in U$.

Although in some ways integral varifolds and rectifiable 
 mod $2$ flat chains are similar, the notions
of convergence are quite different.  Typically (and throughout this paper) convergence of varifolds means weak convergence as radon measures on $U\times {\rm G}_m(\RR^N)$
 (where ${\rm G}_m(\RR^N)$ is the set of $m$-dimensional linear subspaces of $\RR^N$),  and  convergence of flat chains means  convergence with respect to the flat topology
  (see Section~\ref{s:Appendix}). 
  A sequence $V(i)$ of integral varifolds may converge even though the associated flat chains $[V(i)]$ do
not converge.   Similarly, the flat chains $[V(i)]$ may converge even though the varifolds $V(i)$ do not.  Furthermore,
the $V(i)$ and $[V(i)]$ may converge to limits $V$ and $A$, respectively, with $A\ne [V]$.  
See Section~\ref{s:examples}
for examples.

This paper identifies an important situation in which convergence of integral varifolds  implies convergence of
the corresponding mod $2$ flat chains to the expected limit.     In practice, one often proves existence of convergent sequences of integral varifolds by appealing to Allard's compactness theorem
(described in Section~\ref{s:main} below).   Here we prove that if a sequence
of integral varifolds with limit $V$ satisfies the hypotheses of Allard's compactness theorem plus one additional
hypothesis, then the corresponding mod $2$ flat chains converge to $[V]$:

\begin{theorem}\label{Intro:Mod2Theorem}
Let $V(i)$ be a sequence of $m$-dimensional integral varifolds in an open set $U$ of $\RR^N$
that converges to a limit $V$.
Suppose that
\begin{enumerate}[\upshape (1)]
 \item The $V(i)$ satisfy the hypotheses of Allard's compactness theorem for
   integral varifolds, and
 \item The boundaries $\partial [V(i)]$ of the mod $2$ flat chains $[V(i)]$ converge in the flat
 topology.
 \end{enumerate}
Then the chains $[V(i)]$ converge in the flat topology to $[V]$.
\end{theorem}

I do not know whether hypothesis (2) is really necessary.

There is an analogous theorem with rectifiable currents in place of mod $2$ flat chains.  Suppose $A$ is
an $m$-dimensional integer-multiplicity rectifiable current in $U$ and that $V$ is an $m$-dimensional
integral varifold in $U$.   
Recall that $A$ determines an integral varifold $\vv(A)$ by forgetting orientations
    \cite{SimonBook}*{\S 27}.
We say that $A$ and $V$ are {\bf compatible} provided
\[
    V = \vv(A) + 2 W
\]
for some integral varifold $W$ in $U$.  
Thus $A$ and $V$ are compatible if and only if they determine the same 
mod $2$ rectifiable chain.
Equivalently, $A$ and $V$ are compatible provided
\[
    \Theta(V,x) - \Theta(A,x)
\]
is a nonnegative, even integer for $\HH^m$-almost every $x\in U$.

The analog of Theorem~\ref{Intro:Mod2Theorem} for integer-multiplicity currents is the following:

\begin{theorem}\label{Intro:IntegerTheorem}
Let $V(i)$ and $A(i)$ be sequences of $m$-dimensional integral varifolds and integer-multiplicity
currents, respectively, in $U$, such that $V(i)$ and $A(i)$ are compatible for each $i$.
Suppose the $V(i)$ satisfy the hypotheses of Allard's compactness theorem for integral varifolds.
Suppose also that the boundaries $\partial A(i)$ converge (in the integral flat topology) to a limit current.
Then there is a subsequence $i(k)$ such that
the $V(i(k))$ converge to an integral varifold $V$,
the $A(i(k))$ converge to a limit integer-multiplicity current $A$, and
 such that $A$ and $V$ are compatible.
 \end{theorem}
 
The existence of a subsequence for which the limits $V$ and $A$ exist follows immediately
from Allard's compactness theorem for integral varifolds and from the Federer-Fleming compactness
theorem for integer-multiplicity currents.  What is new here is the compatibility of the limits $A$ and $V$.

\section{Preliminaries}

\newcommand{\var}{\operatorname{var}}
\newcommand{\chain}{\operatorname{mod2chain}}
\newcommand{\spt}{\operatorname{spt}}

\stepcounter{theorem}
\subsection{Terminology}\label{s:terminology}
For mod $2$ flat chains, see 
 Fleming's original paper~\cite{Fleming} or, for a different approach, 
  Federer's book~\cite{FedererBook}*{\S4.2.26}.   
Unfortunately (for the purposes of this paper),  a multiplicity $[1]$
plane does not qualify as a mod $2$ flat chain under either 
definition\footnote{Federer's definition requires that a flat chain have compact 
support, and Fleming's definition requires that a flat chain have finite flat norm.}.
By contrast, a multiplicity $1$ plane does qualify as an integral varifold.
Thus in order for the map $V\mapsto [V]$ (as described in the introduction) to 
be a homomorphism from integral varifolds to mod $2$ flat chains,
one must either restrict the class of varifolds or enlarge the class
of flat chains.

If one prefers to restrict, then one should (throughout this paper) replace ``varifold'' by
``compactly supported varifold" and ``flat chain" by 
``compactly supported flat chain".  (Federer's flat chains are automatically compactly supported, but Fleming's need not be.)
Likewise $\LLL(U,\ZZ^+)$ and $\LLL(U,\ZZ_2)$ should be replaced
by the subsets consisting of compactly supported functions.
In particular, the main theorem, 
 Theorem~\ref{Mod2CompatibilityTheorem}, remains true if one makes
 those replacements.   

However, in this paper we have chosen instead to enlarge the class of flat chains.
Fortunately,  only a slight modification in Fleming's definition (or Federer's) 
is required to produce the ``correct'' class of flat chains. 
 (Flat chains so defined would, in the terminology of~\cite{FedererBook}, be called ``locally flat chains" However, although locally flat chains over the integers are briefly mentioned in~\cite{FedererBook} (in Section~4.1.24), 
 the mod $2$ versions are not.)

  See Section~\ref{s:Appendix} for the required modification.  
  
When the coefficient group is the integers (with the standard metric), 
the ``correct'' class of flat chains is defined in~\cite{SimonBook},
and the rectifiability and compactness theorems are proved there.

\stepcounter{theorem}
\subsection{Notation}\label{s:notation}
Suppose $M$ is a Borel subset of a properly embedded $m$-dimensional $C^1$ submanifold of $U$,
or of a countable union of such manifolds.
If $M$ has locally finite $\HH^m$ measure,
we let $[M]$   denote the mod $2$ flat chain associated to $M$ and $\vv(M)$ denote the
integral varifold associated to $M$.
 More generally, if $f: M\to \ZZ^+$
 is a function such that the extension
\begin{align*}
  &F: U\to \ZZ^+ \\
 &F(x) = 
     \begin{cases}
       f(x) &\text{if $x\in M$} \\
       0    &\text{if $x\in U\setminus M$}
      \end{cases}
\end{align*}
is in $\LLL(U,\ZZ^+)$, then we let $\vv(M,f)$ be the 
integral varifold in $U$ corresponding to $F$.

\stepcounter{theorem}
\subsection{Push-forwards}
Suppose that $V$ is an integral varifold in $U$ and that $\phi:U\to W$ is a $C^1$ map
that is proper on $U\cap\spt(\mu_V)$.  Then the push-forward $\phi_\#V$ is also
an integral varifold in $W$ and it satisfies
\begin{equation}\label{e:VarifoldPushForward}
    \Theta(\phi_\#V,y) = \sum_{\phi(x)=y} \Theta(V,x)
\end{equation}
for $\HH^m$-almost every $y\in W$.

Similarly, if $A$ is a rectifiable mod $2$ flat chain in $U$ and if $\phi:U\to W$ is locally lipschitz
on $U\cap\spt{\mu_A}$, then the image chain $\phi_\#A$ satisfies
\begin{equation}\label{e:ChainPushForward}
   [\Theta(\phi_\#A,y)] =   \sum_{\phi(x)=y}{}[\Theta(A,x)]
\end{equation}
for $\HH^m$-almost every $y\in W$.

Note that this determines $\Theta(\phi_\#A,y)$ for $\HH^m$-almost every $y$ since
its value is $0$ or $1$ almost everywhere.  In other words, for $\HH^m$-almost every $y\in W$,
\begin{equation}\label{e:ChainPushForwardCases}
   \Theta(\phi_\#A,y) = 
    \begin{cases}
       1, &\text{if $\sum_{\phi(x)=y}\Theta(A,y)$ is odd, and} \\
       0, &\text{if the sum is even.}
    \end{cases}
\end{equation}

Together~\eqref{e:VarifoldPushForward} and~\eqref{e:ChainPushForwardCases} imply that
\[
    \phi_\#[V] = [\phi_\#V].
\]

We shall need push-forwards only in the special cases where $\phi$ is a dilation or an affine
projection.

\stepcounter{theorem}
\subsection{Examples}\label{s:examples}
Although they are not needed in this paper, some examples illustating the differences between
flat chain convergence and varifold convergence may be instructive.

First, consider a sequence of smooth, simple closed curves $C_i$ lying in a compact region of $\RR^2$
such that the lengths tend to infinity but the enclosed areas tend to $0$.
Let $V_i=\vv(C_i)$ be the corresponding one-dimensional integral varifolds.  
Then the varifolds $V_i$ do not converge, but the corresponding mod $2$ flat chains $[V_i]$
converge to $0$.

Next let 
\begin{equation}\label{e:OddIntervals}
   J_n = \cup \left\{ \left[ \frac{k}{2n}, \frac{k+1}{2n} \right] :  \text{$k$ odd, $0<k<2n$} \right\}
\end{equation}
and let  
\[
    S_n  = J_n \times \{0, 1/(2n)\} \subset \RR^2.
\]
Thus $S_n$ consists of $2n$ horizontal intervals, each of length $1/(2n)$.
Let $V_n=\vv(S_n)$ be the corresponding integral varifold.  
Then the $V_n$ converge to $\vv(I)$, where 
\begin{equation}\label{e:interval}
  I = \{ (x,0): 0\le x \le 1\}.
\end{equation}
However, the corresponding mod $2$ flat chains $[V_n]$ do not converge.
To see this, 
suppose to the contrary that the $[V_n]$ converge to a limit chain $T$.
Let $f, g:\RR^2\to\RR$ be the projections given by
$f(x,y)=x$ and $g(x,y)=x-y$.  Then $f_\#[V_n]=0$ and $g_\#[V_n]=[[0,1]]$.
Passing to the limit, we get
\begin{equation}\label{e:projections}
   f_\#T=0, \qquad g_\#T = [[0,1]].
\end{equation}
However, $T$ is clearly supported in $I$ and $f|I=g|I$, so 
     $f_\#T=g_\#T$
(by~\eqref{e:ChainPushForward}),
contradicting~\eqref{e:projections}.  
This proves that the $[V_n]$ do not converge.
 

For a final example, let
\[
     Q_n = J_n \times [0, (1/n^2)]
\]
where $J_n$ is given by~\eqref{e:OddIntervals}.
Thus $Q_n$ is the union of $n$ closed rectangles, each with base $1/(2n)$ and height $1/(n^2)$.
Let $V_n$ be the one-dimensional varifold associated to the set-theoretic boundary
of $Q_n$: $V_n=\vv(\partial Q_n)$.  Then the $V_n$ converge to $V=\vv(I)$,
where $I$ is given by~\eqref{e:interval}, 
 but the flat chains $[V_n]$ converge to $0$ since the area of $Q_n$ tends to $0$.
 Thus the varifolds $V_n$ converge to $V$ and the chains $[V_n]$ converge to $0$,
 but $[V]\ne 0$.

\section{Proofs of the Main Results}\label{s:main}

Let $V(i)$ be a sequence of $m$-dimensional varifolds in an open subset $U$ of $\RR^N$.
If the $V(i)$ converge to a varifold $V$, then of course
\begin{equation}\label{e:uniformlyboundedmasses}
  \text{$  \limsup \mu_{V(i)} W  <  \infty$ for all $W\subset\subset U$.}
\end{equation}
Conversely, if \eqref{e:uniformlyboundedmasses} holds, then the $V(i)$ have a convergent subsequence (by
the compactness theorem for radon measures.)

\begin{definition}\label{NiceDefinition}
Suppose that $V(i)$, $i=1,2,3,\dots$, and $V$ are $m$-dimensional varifolds in an open subset
$U$ of $\RR^N$.
In this paper, we will say that $V(i)$ converges {\bf with locally bounded first variation} to $V$
provided $V(i)\to V$ as varifolds and 
\begin{equation}\label{e:NicenessBound}
    \limsup_{i\to\infty} \| \delta V(i)\| (W) < \infty
\end{equation}
for every $W\subset\subset U$.
\end{definition}

To understand the definition, the reader may find it helpful to recall that if $V$ is the mutiplicity $1$ varifold associated to a smooth,
embedded manifold-with-boundary $M$, then 
\[
    \| \delta V \| (W) =  \HH^{m-1}(W\cap \partial M)  + \int_{M\cap W} |H(x)|\,d\HH^mx
\]
where $H(x)$ is the mean curvature vector of $M$ at $x$.  Thus for a sequence $V(i)$ of such integral varifolds, the 
  condition~\eqref{e:NicenessBound}  means that the areas of the boundaries and the $L^1$ norms of the mean curvature are uniformly bounded on compact subsets of $U$.
  
(See \cite{AllardFirstVariation} or \cite{SimonBook}*{\S  39} for the general definition
of $\|\delta V\|$.)
  
The following closure theorem of Allard 
(\cite{AllardFirstVariation}*{6.4} or \cite{SimonBook}*{\S42.8})
is one of the key results in the theory of varifolds:

\begin{theorem}\label{AllardClosureTheorem}
If $V(i)$ is a sequence of integral varifolds that converges with locally bounded first variation to $V$, then $V$ is
also an integral varifold.
\end{theorem}

Here we prove:

\begin{theorem}\label{Mod2CompatibilityTheorem}
Suppose $V(i)$ is a sequence of integral varifolds that converge with locally bounded first variation to an integral varifold $V$.
If the boundaries $\partial [V(i)]$ converge (as mod $2$ flat chains) to a limit chain $\Gamma$,
then
\[
     [V(i)] \to [V]
\]
and therefore $\partial [V] = \Gamma$.
\end{theorem}

The last assertion ($\partial [V]=\Gamma$) follows because the boundary operator is continuous
with respect to flat convergence.

The result is already interesting in the case where $\partial [V(i)] =0$ for all $i$.

\begin{proof}
Since $V$ is rectifiable, there is a countable union $\cup\MM$ of $m$-dimensional $C^1$ embedded manifolds
such that
\[
   \mu_V(U\setminus \cup\MM) = 0
\]
Without loss of generality, we may assume that the manifolds in $\MM$ are disjoint.

By the compactness theorem for flat chains of locally finite mass
   (see Theorem~\ref{CompactnessTheorem}), 
a subsequence of the $[V(i)]$ will converge to such a flat chain $A$.  
(Here and throughout the proof, ``flat chain'' means ``mod $2$ flat chain".) 
Using the rectifiability theorem (see Theorem~\ref{RectifiabilityTheorem}),
 we can conclude that $A$ is rectifiable. 

We remark that here one may prove rectifiability of $A$ directly (without invoking 
 Theorem~\ref{RectifiabilityTheorem}).  One sees that as follows.
 By the lower-semicontinuity of mass
with respect to flat convergence, the inequality 
\[
    \mu_{[V(i)]} \le \mu_{V(i)}
\]
implies that
\begin{equation}\label{e:MeasureInequality}
   \mu_A \le \mu_V
\end{equation}
and therefore that
\begin{equation}\label{e:ConcentrationOnM}
  \mu_A(\RR^N\setminus \cup\MM) \le \mu_V(\RR^N\setminus \cup\MM) = 0.
\end{equation}
Hence $A$ is rectifiable.

To show that $A=[V]$, it suffices by~\eqref{e:MeasureInequality}  to show that
\begin{equation}\label{e:DensityInequality}
   \text{$\Theta(\mu_V,x) - \Theta(\mu_A,x)$ is an even integer}  
\end{equation}
for $\mu_V$-almost every $x\in U$.   
By~\eqref{e:ConcentrationOnM}, 
 it suffices to show that \eqref{e:DensityInequality} holds for $\mu_V$-almost every $x\in \cup\MM$.

For $\HH^m$-almost $x\in \cup\MM$ (and therefore in particular for $\mu_V$-almost every 
 $x\in\cup\MM$) we have:
\begin{equation}\label{e:FirstBlowUp}
\begin{split}
  &\eta_{x,\lambda\#}V \to \Theta(V,x) \vv(P), \\
  &\eta_{x,\lambda\#}A \to \Theta(A,x) [P]
\end{split}
\end{equation}
as $\lambda\to 0$, where $P$ is the tangent plane at $x$ to the unique $M\in \MM$ that
contains $x$.  Here $\eta_{x,\lambda}:\RR^N\to\RR^N$ is translation by $-x$ followed by
dilation by $1/\lambda$:
\begin{equation*}
 \eta_{x,\lambda}(y) = \frac1{\lambda}(y-x).
\end{equation*}

The proof of Lemma~42.9 in \ocite{SimonBook} shows that $\mu_V$-almost every $x$ has an additional property, namely
\begin{equation}\label{e:ControlledBoundaryDensity}
 \text{ $\liminf_i \| \delta V(i)\| \BB(x,r) \le cr^m$ for all $r\in (0,1)$}
\end{equation}
where $c=c(x)<\infty$.

We will complete the proof by showing that if $x$ has properties~\eqref{e:FirstBlowUp} 
and~\eqref{e:ControlledBoundaryDensity}, then $\Theta(V,x)$ and
$\Theta(A,x)$ differ by an even integer.

For each fixed $\lambda$, 
\begin{equation}\label{e:FixedLambda}
\begin{split}
  &\eta_{x,\lambda\#}V(i) \to \eta_{x,\lambda\#}V \\
  &\eta_{x,\lambda\#}A(i) \to \eta_{x,\lambda\#}A.
\end{split}
\end{equation}
Thus a standard diagonal argument (applied to~\eqref{e:FirstBlowUp} and~\eqref{e:FixedLambda}) shows that there is a sequence $\lambda(i)\to 0$ such that 
\begin{equation}\label{e:PairConvergence}
\begin{split}
   &\tilde V(i) \to \Theta(V,x) \vv(P) \\
   &[\tilde V(i)] \to \Theta(A,x) [P]
\end{split}
\end{equation}
where
\begin{equation*}
   \tilde V(i) = \eta_{x, \lambda(i)\#}V(i).
\end{equation*}
(One does not need to pass to a subsequence to achieve this. Rather, one simply chooses the 
$\lambda_i$'s to go to $0$ sufficiently slowly.)

Note that $\|\delta V(i)\|\BB(0,r)$ scales like $r^{m-1}$.  Thus~\eqref{e:ControlledBoundaryDensity}
implies that for each $r$,
\[
  \liminf_{i\to \infty} \| \delta \tilde V(i) \| \BB(0,r) = 0.
\]
By passing to a further subsequence, we can assume that the liminf is in fact a limit, so that
\begin{equation}\label{e:VanishingBoundary}
   \| \delta \tilde V(i) \|  \to 0
\end{equation}
as radon measures. (For example, one can choose $i_1 < i_2 < i_3 <\dots$ so that
\[
   \| \delta \tilde V(i_k) \| \BB(0,k) < 1/k
\]
for each $k$.)

Thus we will be done if we can show that~\eqref{e:PairConvergence} 
and~\eqref{e:VanishingBoundary} imply
that $\Theta(V,x)-\Theta(A,x)$ is an even integer.   That is, we have reduced
the theorem to the special case described in the following lemma. 
\end{proof}

\begin{lemma}\label{Lemma}
Suppose
\begin{enumerate}[\upshape (i)]
\item\label{LemmaHypothesis:VarifoldConvergence}
A sequence $V(i)$ of integral varifolds converges to the varifold $V=n\vv(P)$, where  $n$ is a nonnegative integer
and $P$ is an $m$-dimensional linear subspace of $\RR^N$.  
\item\label{LemmaHypothesis:VanishingBoundary} The radon measures $\| \delta V(i) \|$ converge to 
  $0$. 
\item\label{LemmaHypothesis:FlatConvergence}
 The associated mod $2$ flat chains $[V(i)]$ converge to $A= a [P]$, where $a\in \ZZ_2$.
\end{enumerate}
Then $a=[n]$.  
\end{lemma}

\begin{proof}
We may assume that $P=\RR^m\times (0)^{N-m}\subset \RR^N$.  Let 
\begin{equation}
  \pi: \RR^N \cong \RR^m\times \RR^{N-m} \to \RR^m 
\end{equation}
be the orthogonal projection map.

Hypothesis~(\ref{LemmaHypothesis:FlatConvergence}) implies that for almost every $R>0$,
\begin{equation}\label{e:alimit1}
    [V(i)] \llcorner \BB^N(0,R) \to a[P] \cap \BB^N(0,R) = a[P\cap \BB^N(0,R)].
\end{equation}
We can assume that this is the case for $R=1$. (Otherwise dilate by $1/R$.)
We write $\BB$ for $\BB^N(0,R)=\BB^N(0,1)$.

Let $W(i)=V(i)\llcorner \BB$.  By~\eqref{e:alimit1},
\begin{equation}\label{e:alimit2}
   [\pi_\#W(i)] = \pi_\#[W(i)] \to a[\BB^m],
\end{equation}
where $\BB^m=\BB^m(0,1)$.
Also,
\[
     W(i) \to V \llcorner \BB = n \vv(P\cap \BB)
\]
and therefore 
\begin{equation}\label{e:piW}
   \pi_\#W(i) \to n \vv(\BB^m).
\end{equation}
Note that
\begin{equation}\label{e:piWandtheta}
   \pi_\#W(i) = \vv(\BB^m,\theta_i)
\end{equation}
where
\[
   \theta_i(x) = \sum_{y\in \BB \cap \pi^{-1}x} \Theta(W(i),y).
\]
From hypotheses~\eqref{LemmaHypothesis:VarifoldConvergence}
and~\eqref{LemmaHypothesis:VanishingBoundary}, it follows that 
\begin{equation}\label{e:Qsmall}
   \LL^m Q_i \to 0
\end{equation}
where
\begin{equation}\label{e:Qdef}
  Q_i = \{ x\in \BB^m: \theta_i(x) \ne n \}.
\end{equation}
(This is a very nontrivial fact.  Indeed, it is a key part of the proof given in \cite{SimonBook} of the
closure theorem for integral varifolds.  See Remark~\ref{remark} below for a more 
detailed discussion.)

Now
\[
  [\pi_\#W(i)]  = [ \{x\in \BB^m: \text{$ \theta_i(x)$ is odd} \}].
\]
Thus
\[
  [\pi_\#W(i)] - [ n \vv(\BB^m) ] = [ \{x\in \BB^m: \text{$ \theta_i(x) - n$ is odd} \}],
\]
and so (by \eqref{e:Qsmall} and \eqref{e:Qdef})
\[
 M( [\pi_\#W(i)] - [ n\vv(\BB^m)] ) \le \LL^m(Q_i) \to 0.
\]
 Consequently,
\[
   [\pi_\#W(i)] \to [ n\vv(\BB^m)].
\]
This together with \eqref{e:alimit2} implies that $a[\BB^m] = [ n \vv(\BB^m)]$ and thus that 
 $a=[n]$.
\end{proof}

\begin{remark}\label{remark}
Here we elaborate on statement~\eqref{e:Qsmall} of the proof above, because it may not be immediately
apparent to one who reads~\cite{SimonBook} that the lemma we cite (Lemma~42.9) does actually justify that step.
Note that
\begin{equation}\label{e:thetaintegral}
  \int_{\BB^m}\theta_i \to n \LL^m(\BB^m)
\end{equation}
by \eqref{e:piW} and \eqref{e:piWandtheta}.
Let $\eps>0$.  Write
\begin{equation}\label{e:FplusG}
    \theta_i(x) 
      = F_{i,\eps}(x) + G_{i,\eps}(x)
\end{equation}
where
\[
 F_{i,\eps}(x) = \sum \left\{ \Theta(W(i),y): y\in \BB\cap \pi^{-1}(x),\,  |y|<\eps \right\} 
 \]
and
\[
   G_{i,\eps}(x) = \sum \left\{ \Theta(W(i),y) :   y\in \BB\cap \pi^{-1}(x), \, |y|\ge \eps \right\}.
\]

Now
\begin{equation}\label{e:Gto0}
   \int G_{i,\eps} \to 0
\end{equation}
since $W(i)\to n \vv(P)$.  This together with~\eqref{e:thetaintegral} implies that
\begin{equation}\label{e:Fintegral}
  \int F_{i,\eps} \to n \LL^m(\BB^m).
\end{equation}
   
According to \cite{SimonBook}*{Lemma 42.9}, 
\begin{equation}\label{e:Slemma}
  \limsup_{i\to\infty} \int_{\BB^m} (F_{i,\eps} - n)^+\,d\LL^m \le  \omega(\eps)
\end{equation}
for some function $\omega(\cdot)$ such that $\omega(\eps)\to 0$ as $\eps\to 0$.

(Note: there is a mistake in the statement of~\cite{SimonBook}*{Lemma 42.9}: 
instead of~\eqref{e:Slemma}, it asserts the weaker inequality
\[
   \limsup_{i\to\infty} \LL^m\{ x\in \BB^m: F_{i,\eps}(x)>n \}  \le \omega(\eps).
\]
However, the proof of~\cite{SimonBook}*{Lemma 42.9}  establishes the stronger 
statement~\eqref{e:Slemma}.   Indeed the stronger
statement is essential in the proof of
  Allard's integrality theorem~\cite{SimonBook}*{\S42.8}.  In particular,
the stronger statement is used in line (8) of that proof.)

From \eqref{e:Fintegral} and \eqref{e:Slemma}, we see that
\[
   \limsup_{i\to\infty} \int_{\BB^m}|F_{i,\eps} - n |\,d\LL^m \le \omega(\eps)
\]
This together with \eqref{e:Gto0} and  \eqref{e:FplusG} implies that
\begin{equation}\label{e:penultimate}
   \limsup_{i\to\infty}\int_{\BB^m} | \theta_i - n |\,d\LL^m \le \omega(\eps).
\end{equation}
Letting $\eps\to 0$ gives
\[
  \limsup_{i\to\infty} \int_{\BB^m} | \theta_i - n|\,d\LL^m = 0
\]
and thus (since $\theta_i$ is integer-valued)
\[
   \lim_{i\to\infty} \LL^m \{ x\in \BB^m : \theta_i(x)\ne n \} = 0.
\]
\end{remark}

\begin{theorem}\label{IntegerCompatibilityTheorem}
Suppose $V(i)$ is a sequence of integral varifolds that converge with locally bounded first variation to an
integral varifold $V$.  Suppose $A(i)$ is a sequence of integer-multiplicity rectifiable 
currents such
that $V(i)$ and $A(i)$ are compatible.   If the boundaries $\partial A(i)$ converge
(in the integral flat topology) 
to a limit integral flat chain $\Gamma$, then there is a subsequence $i(k)$ such that
the $A(i(k))$ converge to an integer-multiplicity rectifiable current $A$.  Furthermore,
   $V$ and $A$ must then be 
compatible,  and $\partial A$ must equal $\Gamma$.
\end{theorem}

The proof is exactly analogous to the proof of Theorem~\ref{Mod2CompatibilityTheorem}.
Alternatively, one can argue as follows.  
The existence of a subsequence $A(i(k))$ that converges to an integer-multiplicity rectifiable current
$A$ follows from the compactness theorem for such currents (see 
 Theorems~\ref{CompactnessTheorem} and~\ref{RectifiabilityTheorem}).
The ``furthermore'' statement then follows immediately from 
  Theorem~\ref{Mod2CompatibilityTheorem}, together with
the observation that an integral varifold and an integer-multiplicity rectifiable current are 
compatible if and only if they determine the same mod $2$ rectifiable flat chain.

\section{Application to Mean Curvature Flow}

Here we show how the results of this paper rule out certain kinds of singularities
in mean curvature flows.   In another paper, 
we will use similar arguments  to  prove, under mild hypotheses, 
boundary regularity at all times for hypersurfaces moving by mean curvature.  

On both theoretical and experimental grounds, 
grain boundaries in certain annealing metals are believed to move by mean curvature 
flow  \cite{BrakkeBook}*{Appendix A}.   In such metals, one typically sees triple junctions  where three smooth surfaces come together at equal angles along a smooth curve.  Of course one also sees such triple junctions
in soap films, which are equilibrium solutions to mean curvature flow.

Consider the following question:  can an initially smooth surface evolve under mean curvature flow
so as later to develop triple junction type singularities?   More generally, can such a surface have as a blow-up flow
(i.e., a limit of parabolic blow-ups) a static configuration of  $k$ half-planes (counting multiplicity) meeting along a common edge?  
Using Theorem~\ref{Mod2CompatibilityTheorem}, we can (for a suitable formulation of mean curvature flow) prove that the answer is ``no'' if $k$ is odd.

(Suppose $\MM$ is a Brakke flow, $X_i$ is a sequence of spacetime points converging to $X=(x,t)$ with $t>0$, and 
  $\lambda_i$ is a sequence of numbers tending to infinity.  Translate $\MM$ 
in spacetime by $-X_i$ and then dilate parabolically
  by $\lambda_i$ to get a flow $\MM_i$.  
  A {\bf blow-up flow}
of $\MM$ is any Brakke flow that can be obtained as a subsequential limit of such a sequence.)

Let $I\subset \RR$ be an interval, typically either $[0,\infty)$ or all of $\RR$.
Recall that a Brakke flow $t\in I \mapsto V(t)$ of varifolds is called an {\bf integral Brakke flow} provided
$V(t)$ is an integral varifold for almost all $t\in I$.  (See~\cite{BrakkeBook}*{\S3}
or~\cite{IlmanenBook}*{\S6}
for the definition of Brakke flow.)

\begin{definition}
Let $t\mapsto V(t),\, t\in I$ be an integral Brakke flow in $U\subset \RR^N$.
We say that $V(\cdot)$ is {\bf cyclic mod $2$} (or {\bf cyclic} for short) provided $\partial [V(t)] =0$ for almost every $t\in I$.

More generally, suppose $W$ is an open subset of $U$ and $J$ is a subinterval of $I$.  We say that
the Brakke flow $V(\cdot)$ is {\bf cyclic mod $2$ in $W\times J$} if for almost all $t\in J$, 
$[V(t)]$ has no boundary in $W$.  
\end{definition}

We have:

\begin{theorem}\label{theorem:LimitOfCyclicFlows}
Suppose $t\mapsto V_i(t)$ is a sequence of integral Brakke flows that converge as Brakke flows
to an integral Brakke flow $t\mapsto V(t)$.   
If the flows $V_i(\cdot)$ are cyclic mod $2$, then so is the flow 
 $V(\cdot)$.  
 If the flows $V_i(\cdot)$ are cyclic mod $2$ in $W\times J$, then so is the flow 
 $V(\cdot)$.

\end{theorem}

Here convergence as Brakke flows means that for almost all $t$:
\begin{align}
   &\text{$\mu_{V_i(t)} \to \mu_{V(t)}$, and} \\
   \label{e:nice2}&\text{there is a subsequence $i(k)$ (depending on $t$) such that $V_{i(k)}(t) \to V(t)$.}
\end{align}
(This definition may seem peculiar, but this is precisely the convergence that occurs in Ilmanen's compactness theorem for integral Brakke flows \cite{IlmanenBook}*{\S7}.)

Theorem~\ref{theorem:LimitOfCyclicFlows}
follows immediately from Theorem~\ref{Mod2CompatibilityTheorem} and the following lemma 
(which is implicit
in~\cite{IlmanenBook}, but is not actually stated there):

\begin{lemma}
Suppose $t\in I\mapsto V_i(t)$ is a sequence of Brakke flows in $U\subset \RR^N$ that converges
to a Brakke flow $t\mapsto V(t)$.  Then for almost every $t\in I$, there is a subsequence
$i(k)$ such that $V_{i(k)}(t)$ converges with locally bounded first variation to $V(t)$.
Indeed, we can choose the subsequence so that
$\delta V_{i(k)}$ is absolutely continuous with respect to $\mu_{V_{i(k)}}$ and so that
\[
   \sup_{i(k)} \int_{x\in W} |H(V_{i(k)}(t),x)|^2\, d\mu_{V_{i(k)}(t)}x < \infty.
\]
for every $W\subset \subset U$, 
where $H(V_{i(k)}(t),\cdot)$ is the generalized mean curvature of $V_{i(k)}(t)$.
\end{lemma}

\begin{proof}  For simplicity, let us assume that $I=[0,\infty)$. 

Recall that for almost all $t$, the varifold $V_i(t)$ has bounded first variation, and the singular
part of the first variation measure is $0$.  Thus (for such $t$) 
\begin{equation}\label{e:CauchySchwartz}
     \| \delta V_i(t)\| (W) 
       = \int_W |H_{i,t}|\,d\mu_{i,t}  
       \le \left( \int_W |H_{i,t}|^2\,d\mu_{i,t}\right)^{1/2}
            \left(  \vphantom{\int} \mu_{i,t}(W)  \right)^{1/2}.
\end{equation}
where $H_{i,t}$ is the generalized mean curvature of $V_i(t)$ and $\mu_{i,t}=\mu_{V_i(t)}$.

Consider first the case that the varifolds $V_i(t)$ are all supported in some compact set.
Then the initial total masses $\M(V_i(0))=\mu_{V_i(0)}(\RR^N)$ are bounded above by some $C<\infty$.
Since mass decreases under mean curvature flow, the same bound holds for all $t>0$.
By definition of Brakke flow,
\[
   \overline{D}_t \M(V_i(t))  \le - \int |H(V_i(t),\cdot)|^2\,d\mu_{V_i(t)},
\]
so
\begin{equation}\label{e:L2boundonH}
   \int_{t\in I} \int |H(V_i(t), \cdot)|^2 \,d\mu_{V_i(t)}\,dt  \le C.
\end{equation}
Thus by Fatou's theorem,
\[
   \int_{t\in I}\left(  \liminf_i \int |H(V_i(t),\cdot)|^2\,d\mu_{V_i(t)} \right) \,dt \le C
\]
In particular, 
\[
   \liminf_i \int |H(V_i(t),\cdot)|^2\,d\mu_{V_i(t)}  < \infty
\]
for almost every $t$.  For each such $t$, there is a subsequence $i(k)$ such that
\[
   \sup_{k}  \int |H(V_{i(k)}(t),\cdot)|^2\,d\mu_{V_{i(k)}(t)}  < \infty
\]
This together with~\eqref{e:CauchySchwartz} implies that the $V_{i(k)}(t)$ converge
with locally bounded first variation to $V(t)$ (in the sense of Definition~\ref{NiceDefinition}).

The general case (noncompactly supported varifolds) is essentially the same, except that instead
of \eqref{e:L2boundonH} one uses the local bound:
\[
      \sup_i    \int_{t\in J} \int_{x\in W} |H(V_i(t),x)|^2\,d\mu_{V_i(t)}\,dt  < \infty
\]
together with the mass bound
\[
   \sup_i \sup_{t\in J} \mu_{V_i(t)}(W) < \infty,
\]
both of which bounds hold for all intervals $J\subset\subset I$ 
and open subsets $W\subset\subset U$ \cite{EckerBook}*{Proposition 4.9}.
\end{proof}

\begin{remark}\label{StrongerRemark}
The lemma and Allard's closure theorem (Theorem~\ref{AllardClosureTheorem})
imply that
a limit of integral Brakke flows is also integral.   
\end{remark}

\begin{corollary}\label{corollary:NoOddJunctions}
Suppose $k$ is an odd integer.
A static configuration of $k$-half planes (counting multiplicity) meeting along a 
common edge cannot occur as a blow-up flow
to an integral Brakke flow that is cyclic mod $2$.
\end{corollary}

\begin{proof}
Let $V$ be the varifold corresponding to $k$ 
halfplanes (counting multiplicity) meeting along an edge $E$.
If the static flow $t\mapsto V$ is a limit flow to an integral
Brakke flow that is cyclic mod $2$, then this static flow is also cyclic mod $2$ and thus $\partial [V]=0$.
But $\partial [V]$ is the common edge $E$
with multiplicity $[k]$, so $k$ must then be even.
\end{proof}

The following theorem shows that for rather arbitrary initial surfaces,  there exist nontrivial integral Brakke flows that are cyclic mod $2$.

\begin{theorem}\label{theorem:Mod2FlowExistence}
 Let $A_0$ be any compactly supported rectifiable mod $2$ cycle in $\RR^N$. 
  (For example, $A_0$ could be the 
mod $2$ rectifiable flat chain associated to a $C^1$ compact, embedded submanifold.) 
Then there is an integral Brakke flow $t\in [0,\infty)\mapsto V(t)$ and a one-parameter family
$t\in [0,\infty)\mapsto A(t)$ of rectifiable mod $2$ flat chains with the following properties:
\begin{enumerate}[\upshape (1)]
 \item\label{FirstConclusion} $A(0)=A_0$ and $\mu_{V(0)} = \mu_{A(0)}$.
 \item $\partial A(t)=0$ for all $t$.
 \item $t\mapsto A(t)$ is continuous with respect to the flat topology.
 \item\label{PenultimateConclusion} $\mu_{A(t)} \le \mu_{V(t)}$ for all $t$.
 \item\label{AandVcompatible} $A(t)=[V(t)]$ for almost every $t$.
 \end{enumerate}
\end{theorem}

In particular, the flow is cyclic mod $2$, and thus triple (or more generally odd-multiplicity) junctions
cannot occur in $V(\cdot)$ by Corollary~\ref{corollary:NoOddJunctions}.

(Remark about assertion~\eqref{AandVcompatible}: 
Since $V(\cdot)$ is an integral Brakke flow, $V(t)$ is an integral
varifold for almost all $t$ and thus $[V(t)]$ is well-defined
for almost all $t$.)

\begin{proof}
Except for assertion~\eqref{AandVcompatible}, this was proved 
by~Ilmanen~\cite{IlmanenBook}*{8.1 and 8.3}.   
He used integer-multiplicity currents rather than
 mod $2$ flat chains, but his proof works equally
well in either context.   (The $A(t)$ here is the slice $T_t$ in Ilmanen's notation.)
The flat continuity~(3) is not stated there, but it follows immediately from \cite{IlmanenBook}*{8.3}.

Roughly speaking, Ilmanen constructs $V(\cdot)$ and $A(\cdot)$ as limits of  ``nice" examples
$V_i(\cdot)$ and $A_i(\cdot)$ for which
\[
   \mu_i(t) = \mu_{A(i)}(t)
\]
for all $t$.

Now his $A_i(t)$ are not quite cycles.  
However, $A_i(t)$ moves by translation, and it moves very fast if $i$ is large.
In particular, if $U\subset\subset \RR^N$ and 
 $I\subset\subset (0,\infty)$, then for sufficiently large $i$ and for all $t\in I$,
 $\partial A_i(t)$ lies outside $U$.   

 Thus (exactly as in the proof of Theorem~\ref{theorem:LimitOfCyclicFlows},
 or by Remark~\ref{StrongerRemark} and Theorem~\ref{Mod2CompatibilityTheorem}),
  we deduce (for almost every $t\in I$) that 
     $A(t)\llcorner U = [V(t)]\llcorner U$
 and that $\partial[V(t)]$ lies outside $U$.
 
 Since $U$ is arbitrary, this gives~\eqref{AandVcompatible}.
 \end{proof}
 
 \begin{remark}
 The description just given is a slightly simplified account of Ilmanen's proof.  Actually he does not quite
 get the pair ($V(\cdot), A(\cdot))$ as limits of  nice examples.  Rather he gets a pair of flows
 $(\mu^*(\cdot), A^*(\cdot))$ of one higher dimension
 as such a limit.  
 The argument given above shows that $(\mu^*(\cdot), A^*(\cdot))$ has the property corresponding
 to property~\eqref{AandVcompatible} above (and Ilmanen in his proof shows that it has 
 properties~\eqref{FirstConclusion}-\eqref{PenultimateConclusion}.
 Now the pair $(\mu^*(\cdot), A^*(\cdot))$ is translation
 invariant in one spatial direction.   By slicing, Ilmanen gets the desired pair $(\mu(\cdot), A(\cdot))$.
 Translational invariance implies (in a straightforward way) that 
  properties~\eqref{FirstConclusion}-\eqref{AandVcompatible}
   for $(\mu(\cdot), A(\cdot))$
   are equivalent to the
 corresponding properties for $(\mu^*(\cdot), A^*(\cdot))$.
 \end{remark}
 
Theorem~\ref{theorem:Mod2FlowExistence} has an analog for integer-multiplicity currents in place
of mod $2$ flat chains:

\begin{theorem}\label{theorem:IntegralFlowExistence}
Let $A_0$ be any compactly supported integer-multiplicity cycle (i.e., integer-multiplicity current
with $\partial A_0=0$.)   
Then there is an integral Brakke flow $t\in [0,\infty)\mapsto V(t)$ and a one-parameter family
$t\in [0,\infty)\mapsto A(t)$ of integer-multiplicity currents with the following properties:
\begin{enumerate}[\upshape (1)]
 \item $A(0)=A_0$ and $\mu_{V(0)} = \mu_{A(0)}$.
 \item $\partial A(t)=0$ for all $t$.
 \item $t\mapsto A(t)$ is continuous with respect to the flat topology.
 \item $\mu_{A(t)} \le \mu_{V(t)}$ for all $t$.
 \item $A(t)$ and $V(t)$ are compatible for almost every $t$.
 \end{enumerate}
\end{theorem}

We omit the proof since it is almost identical to the proof of the mod $2$ case, 
Theorem~\ref{theorem:Mod2FlowExistence}

Note that if an integer-multiplicity current $A$ is compatible with an integral varifold $V$, then
$[V]$ is the flat chain mod $2$ corresponding to $A$.  It follows that the Brakke flow $V(\cdot)$
in Theorem~\ref{theorem:IntegralFlowExistence}
is cyclic mod $2$.  In particular, triple (or more generally odd-multiplicity) junctions
cannot occur in $V(\cdot)$ by Corollary~\ref{corollary:NoOddJunctions}.

Ruling out even-multiplicity junctions is more subtle.   In particular, limits of smooth Brakke
flows can have quadruple junctions.   For example, recall that Sherk constructed a complete, embedded,
singly periodic minimal surface in $\RR^3$ that is, away from the $z$-axis, asymptotic to the union
of the planes $x=0$ and $y=0$.   We may regard that surface as an equilibrium solution to mean
curvature flow.   Now dilate by $1/n$ and let $n\to \infty$.  The limit surface is a pair of orthogonal
planes and thus has a quadruple junction.

\section{Appendix: Flat Chains}\label{s:Appendix}

Let $G$ be a metric abelian coefficient group, i.e., an abelian
group with a translation invariant metric $d(\cdot,\cdot)$.
The norm $|g|$ of a group element $g$ is defined to be its distance from $0$.
The groups relevant for this paper are $\ZZ_2$ and $\ZZ$, both
with the standard metrics.
If $U$ is an open subset of $\RR^N$, 
let $\FF_c(U;G)$ be the space of flat chains
with coefficients in $G$ and with compact support in $U$, as defined in \cite{Fleming}.
We let $\FF_{m,c}(U;G)$ denote the space of $m$-dimensional chains in $\FF_c(U;G)$.

If $W$ is an open subset of $\RR^N$ and $A\in \FF_c(\RR^N;G)$, we let 
$\M_W(A)$ be the minimum of 
\begin{equation}\label{e:MassSeminorm}
    \liminf \mu_{A(i)}(W)
\end{equation}
among all sequences of compactly supported, finite-mass flat chains $A(i)$ such that
$A(i)$ converges in the flat topology to $A$.
By lower-semicontinuity of mass, $\M_W(A)=\mu_A(W)$ for any
chain $A$ of finite mass.

We define the flat seminorm $\FF_W$ by
\[
   \FF_W(A) = \inf \{ M_W(A-\partial Q) + M_W(Q) \},
\]
where the infimum is over all $Q\in \FF_{c}(\RR^N;G)$.

Let $U$ be an open subset of $\RR^N$.   Choose a countable collection $\mathcal{W}$  of nested open sets whose union is 
$U$ and each of whose closures is a compact subset of $U$.
We define the space $\FF_m(U;G)$ of flat $m$-chains in $U$ with 
coefficients in $G$ to be the completion of $\FF_{m,c}(U;G)$ with 
respect to the seminorms $\FF_W$ for $W\in \mathcal{W}$.
(It is straightforward to show that the resulting space is independent of the choice of
 $\mathcal{W}$.)
 
 By continuity, the seminorms $\FF_W$ extend to all of $\FF_m(U;G)$.
We also define the mass seminorms $\M_W$ on all of $\FF_m(U;G)$
exactly as above~\eqref{e:MassSeminorm}.

Convergence of flat chains means flat convergence, i.e., convergence
with respect to the seminorms $\FF_W$ for all
open $W\subset\subset U$ or, equivalently, for all  $W\in \mathcal{W}$
for a collection $\mathcal{W}$ of nested open sets as above.

We define the support of a flat chain $A\in \FF_m(U;G)$ as follows:
$x\notin \spt A$ if and only if there is a sequence $A_i\in \FF_{m,c}(U;G)$
and a ball $\BB(x,r)$ such that $A_i\to A$ and such that $\spt A_i$
is disjoint from $\BB(x,r)$ for every $i$.

In the proof of the main results, Theorems~\ref{Mod2CompatibilityTheorem} 
 and~\ref{IntegerCompatibilityTheorem}, we used the following version of the compactness
theorem for flat chains.   It is valid for any coefficient group $G$ in which all sets
  of the form
$\{ g\in G: |g| \le r\}$) are compact.   In particular, it is valid for the integers with
the usual norm and for the integers mod $2$.

\begin{theorem}[Compactness Theorem]\label{CompactnessTheorem}
Let $A_i$ be a sequence of flat $m$-chains in $U$ such that
the boundaries $\partial A_i$ converge to a limit chain $\Gamma$, and such that
\begin{equation}\label{MassBound}
   \limsup_i \M_W(A_i) < \infty
\end{equation}
for every open $W\subset\subset U$.
Then $A_i$ has a convergent subsequence.
\end{theorem}

We first prove the version for compact supports:

\begin{lemma}\label{CompactnessLemma}
If $A_i, \Gamma\in \FF_c(\RR^N;G)$ are supported in a fixed compact subset $X$
of $\RR^N$, if $\sup_i\M(A_i)<\infty$, and if $\FF(\partial A_i - \Gamma)\to 0$, then
there is a subsequence $A_{i(k)}$ and a chain $A$ such that $\FF(A_{i(k)}-A)\to 0$.
\end{lemma}

\begin{proof}
We may assume $X$ is convex (otherwise replace it by its convex hull.)
Since $\partial A_i\to \Gamma$, we have $\partial \Gamma=0$.  It follows that
$\Gamma=\partial R$ for some chain $R$ of finite mass.

By hypothesis, 
\[
   \FF( \partial A_i - \partial R) \to 0.
\]
Thus there are chains $Q_i$ such that
\begin{equation}\label{zapped}
  \M(Q_i) + \M ( \partial Q_i + \partial A_i - \partial R)   \to 0.
\end{equation}
We may assume that $R$ and the $Q_i$ are supported in $X$. (Otherwise map them into $X$ by the nearest point
retraction of $\RR^N$ to $X$.)

Now let 
\[
  A_i^* = Q_i + A_i - R.
\]
Note that 
 \[
    \limsup_i \M(A_i^*) \le \sup_i\M(A_i) + \M(R)<\infty
 \]
 since $\M(Q_i)\to 0$ by~\eqref{zapped}.
From~\eqref{zapped} we also see that $\M(\partial A_i^*)\to 0$,
so in particular $\sup_i \M(\partial A_i^*)<\infty$.

Thus by the standard compactness theorem (see for example~\cite{Fleming}*{7.4}), 
we may, by passing to a subsequence,
assume that the $A_i^*$ converge to a limit $A^*$.
Hence
\begin{align*}
  \FF(A_i - (A^*+R)) 
  &= \FF( A_i^* - A^* - Q_i) \\
  &\le \FF(A_i^*-A^*) + \FF(Q_i) \\
  &\le \FF(A_i^*-A^*) + \M(Q_i) \\
  &\to 0
\end{align*}
since $A_i^*\to A^*$ and $ \MM(Q_i) \to 0$.
Thus the $A_i$ converge to $A^*+R$.
\end{proof}

\begin{proof}[Proof of Theorem~\ref{CompactnessTheorem}]
Let $W$ be an open set whose closure is a compact subset of $U$.
Choose an open set $V$ whose interior contains the closure of $W$ and whose closure
is a compact subset of $V$.
The idea of the proof is to work in a one-point compactification of $V$ so that we can apply
Lemma~\ref{CompactnessLemma}.

Let $u: \RR^N\to [0,1]$ be a smooth function that is $1$ on $W$, that is strictly positive on $V$,
and that vanishes on $\RR^N\setminus V$.

Define $f: \RR^N \to \RR^{N+1}$ by
\[
   F(x) = u(x) (x,1).
\]
Note that $f$ is lipschitz and that $f$ maps the complement of $V$ to a point.
(Indeed, $f(\RR^N)$ may be regarded as a one-point compactification of $V$.)
It follows that  $\M(f_\#S)$ and $\FF(f_\#S)$ can
be bounded by a constant times $\MM_V(S)$ and $\FF_V(S)$, respectively.

Let $A_i^*= f_\#A_i$.  Then the hypotheses of Lemma~\ref{CompactnessLemma} are satisfied
for the $A_i^*$.  Thus by passing to a subsequence we may assume that the $A_i^*$
converge in the $\FF$ metric.

By passing to a further subsequence, we may assume that 
\begin{equation}\label{SumOfFlatNorms}
   \sum_i \FF(A_i^* - A_{i+1}^*) < \infty.
\end{equation}
Let $H=H_\zeta$ be a halfspace of the form $\RR^N\times [\zeta,\infty)$.
From~\eqref{SumOfFlatNorms} it follows that
\begin{equation}\label{SlicedFlatSum}
  \sum_i \FF( A_i^*\llcorner H - A_{i+1}^*\llcorner H) < \infty
\end{equation}
for almost every $\zeta$ (See~\cite{Fleming}*{Lemma 2.1}). 
 Fix such a $\zeta\in (0,1)$ and the corresponding $H$.

The radial projection map
\begin{align*}
  &\pi: H\to \RR^N \\
  &\pi(x,y) = x/y
\end{align*}
is lipschitz, so by~\eqref{SlicedFlatSum} the chains $A_i^\dag:=\pi_\#(A_i^*\llcorner H)$ are 
 $\FF$-convergent.

It follows that the $A_i^\dag$ are also $\FF_W$ convergent (since $\FF_W\le \FF$).

But $\pi\circ f$ is the identity on $W$.  Hence $A_i^\dag$ and $A_i$ coincide in $W$.
(In other words, $A_i - A_i^\dag$ is supported in $W^c$.)

Thus the $A_i$ are also $\FF_W$ convergent.

We have shown that for every open $W\subset\subset U$, there is an $\FF_W$-convergent
subsequence of the $A_i$.  Now apply
 the diagonal argument to a nested sequence of such $W$'s that exhaust $U$.
\end{proof}

\begin{corollary}
Suppose $A_i$ are flat chains in $U$ such that 
\[
  \limsup_i (\M_W(A_i) + \M_W(\partial A_i)) < \infty
\]
for every $W\subset\subset U$.  Then $A_i$ has a subsequence
that converges in the flat topology.
\end{corollary}

\begin{proof}
By Theorem~\ref{CompactnessTheorem} applied to the $\partial A_i$, there is a subsequence $i(k)$
for which the boundaries $\partial A_{i(k)}$ converge. 
Consequently the $A_{i(k)}$ satisfy the hypotheses of Theorem~\ref{CompactnessTheorem}.
\end{proof}

\begin{theorem}[Rectifiablity Theorem]\label{RectifiabilityTheorem}
Suppose $A$ is a flat $m$-chain in $U$ with locally finite mass.
Then $A$ is rectifiable.
\end{theorem}

Of course ``$A$ has locally finite mass'' means ``$\M_W(A)<\infty$
for every open $W\subset\subset U$".

The theorem was proved in the case $G=\ZZ$ by Federer and Fleming~\cite{FedererFleming}.
The proof is also presented in~\cite{FedererBook} and in~\cite{SimonBook}.
Rather different proofs are given in~\cite{SolomonClosure}
and~\cite{WhiteCompactness}.
Fleming proved the rectifiabilty theorem for all finite coefficient groups~\cite{Fleming}.
For the most general result, see~\cite{WhiteRectifiability},
which gives a simple necessary and sufficient condition on the coefficient group
in order for the rectifiablity theorem to hold.

\end{document}